\documentclass[12pt]{amsart}
\usepackage{graphicx}
\usepackage{amssymb}
\usepackage{tikz}
\usetikzlibrary{matrix,arrows.meta,bending}
\usepackage{color}
\usepackage[numbers]{natbib}

\usepackage[a4paper, left=3cm, right=3cm, top=3cm, bottom=3cm, footskip=15mm]{geometry}

\newtheorem{theorem}{Theorem}[section]

\newtheorem{lemma}[theorem]{Lemma}
\newtheorem{proposition}[theorem]{Proposition}
\newtheorem{corollary}[theorem]{Corollary}

\theoremstyle{definition}
\newtheorem{definition}[theorem]{Definition}

\theoremstyle{remark}
\newtheorem{remark}[theorem]{Remark}

\numberwithin{equation}{section}

\makeindex

\usepackage{fancyhdr}
\usepackage{calc} 

\AtBeginDocument{%
  \setlength{\textwidth}{\dimexpr\paperwidth - 6cm\relax}%
  \setlength{\oddsidemargin}{\dimexpr 3cm - 1in\relax}%
  \setlength{\evensidemargin}{\dimexpr 3cm - 1in\relax}%
  \setlength{\marginparwidth}{2.5cm}%
}

\begin{document}

\title{On the general no-three-in-line problem}

\author{T. Agama}
\address{Department of Mathematics, African Institute for Mathematical science, Ghana
}
\email{theophilus@aims.edu.gh/emperordagama@yahoo.com}


\date{\today}



\begin{abstract}
In this paper, we show that the number of points that can be placed in the grid $n\times n\times \cdots \times n~(d~times)=n^d$ for all $d\in \mathbb{N}$ with $d\geq 2$ so that no three points are collinear satisfies the lower bound
\begin{align}
\gg n^{d-1}\sqrt[2d]{d}.\nonumber
\end{align}
This extends the result of the no-three-in-line problem to all dimension $d\geq 3$.
\end{abstract}

\maketitle

\section{Introduction}

The classical \emph{no-three-in-line} problem asks for the largest cardinality of a subset of the integer grid
$$
\{1,\dots,n\}\times\{1,\dots,n\}
$$
containing no three collinear points. Posed in recreational form by H. Dudeney and subsequently studied in a rigorous combinatorial setting, this question sits at the crossroads of discrete geometry and additive combinatorics and has motivated a range of constructive and extremal techniques. Roth and Erd\H{o}s made early asymptotic progress using density and combinatorial arguments; in particular, the methods of Roth and related combinatorial constructions give nontrivial lower bounds for large $n$ (see \cite{roth1951problem}). Subsequent refinements and explicit constructions improved these results in the planar case (see \cite{hall1975some}), and several authors have proposed natural conjectures for the optimal linear growth constant (see \cite{pegg2005math} and references therein). A natural higher-dimensional extension was studied by P\'or and Wood, who showed that in three dimensions one can place $\Theta(n^2)$ points in an $n\times n\times n$ grid with no three collinear \cite{por2007no}.\\

In this article, we extend the no-three-in-line problem to arbitrary dimension $d\geq 2$ and give an explicit constructive lower bound for the maximal number of lattice points in the $d$-dimensional box
$$
\underbrace{\{1,\ldots,n\}\times\cdots\times\{1,\ldots,n\}}_{d\text{ factors}}=:[n]^d
$$
that contain no three collinear points. Our main theorem asserts that one can select at least
$$
\gg n^{d-1}d^{1/(2d)}
$$
such points (the implied constant absolute), thereby extending the three-dimensional $\Theta(n^2)$ phenomenon to all higher dimensions with a quantitative dependence on $d$. The construction and counting arguments are elementary in spirit but take advantage of a new geometric viewpoint based on the \emph{compression method} introduced in \cite{agama2019compression}.\\

At the heart of the approach is a rescaling (reciprocal) map - the compression $\mathbb{V}_m$ of scale $m\in(0,1]$ - which acts coordinatewise by reciprocation against a fixed scale. Two closely related scalar statistics associated to a compressed point, which we call the \emph{mass} and the \emph{compression gap}, measure respectively the summed reciprocals of coordinates and a symmetric deviation between a point and its reciprocal image. These statistics permit two key geometric steps used throughout the paper\\

\begin{enumerate}

\item By fixing a target compression-gap magnitude one obtains a family of \emph{induced balls} whose boundary (the set of \emph{admissible} points) has the combinatorial property that no three admissible points are collinear. This geometric noncollinearity is a simple but robust property of the admissible locus (see Proposition \ref{No three collinear}).
\bigskip

\item Lattice-point counting on the smallest axis-parallel box that contains an induced ball produces the desired lower bound: roughly speaking, admissible points yield a large supply of lattice points distributed along a $(d-1)$-dimensional skeleton of the box, and a careful estimate of the local density yields the $n^{d-1}d^{1/(2d)}$ term.
\end{enumerate}
\bigskip

The compression viewpoint gives two advantages. First, it provides a uniform mechanism to build noncollinear configurations in every dimension by working with a single geometric object (an induced ball) and sampling its admissible lattice points. Second, because the mass and compression-gap are explicit analytic functions of the coordinates, one can transport coarse bounds on coordinate sizes to sharp counting estimates on admissible lattice points; this is the analytic core of the counting argument (see Proposition \ref{cgidentity} and Lemma \ref{gapestimate}).\\

\subsection{Organization of the paper.} In Section \ref{sec:prelim}, we introduce notation, recall the compression map $\mathbb{V}_m$ and its basic algebraic properties, and record elementary harmonic sums used later (Lemma \ref{elementary}). Section \ref{sec:massgap} defines the mass and the compression gap and establishes the upper and lower estimates that relate these quantities to the coordinate extrema (Proposition \ref{crucial} and Corollary \ref{compression gap estimate 2}). In Section \ref{sec:balls}, we define the ball induced by a point under compression, develop the nesting and limit-point structure of these balls (Theorems \ref{decider}, \ref{ballproof}, and \ref{limitexistence}), and introduce admissible points on induced balls together with their combinatorial properties (Theorem \ref{admissibletheorem} and Proposition \ref{No three collinear}).  Section \ref{sec:main} contains the counting argument and the proof of the main theorem; two short corollaries for $d=2$ and $d=3$ are stated at the end.\\

\textbf{Remarks and constants.} The method is constructive - admissible points are explicitly described - and the hidden constant in the $\gg$ notation is absolute and independent of $n$ (it may depend mildly, and only in auxiliary estimates, on negligible choices of the compression scale ($m=d\mapsto m(d)=o(1)$). We do not claim that the exponent on $d$ is optimal; rather, the bound exhibits the correct $n^{d-1}$ leading growth and supplies an explicit, dimension-dependent multiplicative factor that naturally emerges from the compression geometry.\\

\textbf{Notation.} Throughout $f(n)\gg g(n)$ means that there exists $c>0$ with $f(n)\geq c\,g(n)$ for all large $n$; subscripts indicate dependence on parameters when necessary (for example $f(n)\gg_s g(n)$). The asymptotic notation $\sim$ and $o(1)$ are standard.
\bigskip

\section{Preliminary results}\label{sec:prelim}

\begin{definition}
By compression of the fixed scale $m\in\mathbb{R}$ with $1\geq m>0$ on the points in $\mathbb{R}^{n}$, we mean the map $\mathbb{V}:(\mathbb{R}\setminus \{0\})^n\longrightarrow \mathbb{R}^n$ such that 
\begin{align}
\mathbb{V}_m[(x_1,x_2,\ldots, x_n)]=\bigg(\frac{m}{x_1},\frac{m}{x_2},\ldots,\frac{m}{x_n}\bigg)\nonumber
\end{align}
for $n\geq 2$ and with $x_i\neq x_j$ for $i\neq j$ and  $x_i\neq 0$ for all $i=1,\ldots,n$. 
\end{definition}

\begin{remark}
The notion of compression is, in some way, the process of rescaling points in $\mathbb{R}^n$ for $n\geq 2$. Thus, it is important to observe that a compression roughly speaking pushes points very close to the origin away from the origin by a certain scale and similarly draws points that are away from the origin close to the origin. Intuitively, compression induces some kind of motion on points in the Euclidean space $\mathbb{R}^n$ for $n\geq 2$.
\end{remark}

\begin{proposition}
A compression $\mathbb{V}_m:(\mathbb{R}\setminus 0)^n\longrightarrow \mathbb{R}^n$ of scale $m$ with $1\geq m>0$ is a bijective map.
\end{proposition}

\begin{proof}
Suppose that $\mathbb{V}_m[(x_1,x_2,\ldots, x_n)]=\mathbb{V}_m[(y_1,y_2,\ldots,y_n)]$. We get 
\begin{align}
\bigg(\frac{m}{x_1},\frac{m}{x_2},\ldots,\frac{m}{x_n}\bigg)=\bigg(\frac{m}{y_1},\frac{m}{y_2},\ldots,\frac{m}{y_n}\bigg).\nonumber
\end{align}
We deduce $x_i=y_i$ for each $i=1,2,\ldots, n$. Surjectivity follows by the definition of the map.
\end{proof}

\subsection{The mass of compression}\label{sec:massgap}

In this section, we introduce the notion of the mass of compression on points in space and study the associated statistics.

\begin{definition}\label{mass}
By the mass of a compression $\mathbb{V}_m$ of fixed scale $m$ with $1\geq m>0$, we mean the map $\mathcal{M}:\mathbb{R}^n\longrightarrow \mathbb{R}$ such that 
\begin{align}
\mathcal{M}(\mathbb{V}_m[(x_1,x_2,\ldots,x_n)])=\sum \limits_{i=1}^{n}\frac{m}{x_i}.\nonumber
\end{align}
\end{definition}
\bigskip

\begin{lemma}\label{elementary}
We have 
\begin{align}
\sum \limits_{n\leq x}\frac{1}{n}=\log x+\gamma+O\bigg(\frac{1}{x}\bigg)\nonumber
\end{align}
where $\gamma=0.5772\cdots $.
\end{lemma}
\bigskip

We prove an upper and lower bound for the mass of the compression of fixed scale $m$ with $1\geq m>0$.

\begin{proposition}\label{crucial}
Let $(x_1,x_2,\ldots,x_n)\in \mathbb{R}^n$ with $x_i\neq 0$ for each $1\leq i\leq n$ and $x_i\neq x_j$ for $i\neq j$. We have
\begin{align}
m\log \bigg(1-\frac{n-1}{\mathrm{sup}(x_j)}\bigg)^{-1} \ll \mathcal{M}(\mathbb{V}_m[(x_1,x_2,\ldots, x_n)])\ll m\log \bigg(1+\frac{n-1}{\mathrm{Inf}(x_j)}\bigg)\nonumber
\end{align}
for $n\geq 2$.
\end{proposition}

\begin{proof}
Let $(x_1,x_2,\ldots,x_n)\in \mathbb{R}^n$ for $n\geq 2$ with $x_j\neq 0$. We get
\begin{align}
\mathcal{M}(\mathbb{V}_m[(x_1,x_2,\ldots, x_n)])&=m\sum \limits_{j=1}^{n}\frac{1}{x_j}\nonumber \\&\leq m\sum \limits_{k=0}^{n-1}\frac{1}{\mathrm{Inf}(x_j)+k}.\nonumber
\end{align}
We deduce for the lower bound
\begin{align}
\mathcal{M}(\mathbb{V}_m[(x_1,x_2,\ldots, x_n)])&=m\sum \limits_{j=1}^{n}\frac{1}{x_j}\nonumber \\&\geq m\sum \limits_{k=0}^{n-1}\frac{1}{\mathrm{sup}(x_j)-k}.\nonumber
\end{align}
\end{proof}

\begin{definition}\label{gap}
Let $(x_1,x_2,\ldots, x_n)\in \mathbb{R}^n$ with $x_i\neq 0$ for all $i=1,2\ldots,n$. By the \emph{compression gap} $\mathcal{G}\circ \mathbb{V}_m[(x_1,x_2,\ldots, x_n)]$ of the scale $m>0$, we mean the measure 
\begin{align}
\mathcal{G}\circ \mathbb{V}_m[(x_1,x_2,\ldots, x_n)]=\bigg|\bigg|\bigg(x_1-\frac{m}{x_1},x_2-\frac{m}{x_2},\ldots,x_n-\frac{m}{x_n}\bigg)\bigg|\bigg|.\nonumber
\end{align}
\end{definition}
\bigskip

\section{The ball induced by compression}\label{sec:balls}

In this section, we introduce the notion of the ball induced by a point $(x_1,x_2,\ldots,x_n)\in \mathbb{R}^n$ under compression of a given scale.

\begin{definition}
Let $(x_1,x_2,\ldots,x_n)\in \mathbb{R}^n$ with $x_i\neq x_j$ for all $1\leq i<j\leq n$ and $x_i\neq 0$ for all $1\leq i\leq n$. By the ball induced by $(x_1,x_2,\ldots,x_n)\in \mathbb{R}^n$ under compression of the scale $1\geq m>0$, denoted by $\mathcal{B}_{\frac{1}{2}\mathcal{G}\circ \mathbb{V}_m[(x_1,x_2,\ldots, x_n)]}[(x_1,x_2,\ldots,x_n)]$, we mean the inequality
\begin{align}
\left|\left|\vec{y}-\frac{1}{2}\bigg(x_1+\frac{m}{x_1},x_2+\frac{m}{x_2},\ldots,x_n+\frac{m}{x_n}\bigg)\right|\right|<\frac{1}{2}\mathcal{G}\circ \mathbb{V}_m[(x_1,x_2,\ldots, x_n)].\nonumber
\end{align}
A point $\vec{z}=(z_1,z_2,\ldots,z_n)\in \mathcal{B}_{\frac{1}{2}\mathcal{G}\circ \mathbb{V}_m[(x_1,x_2,\ldots, x_n)]}[(x_1,x_2,\ldots,x_n)]$ if it satisfies the inequality.
\end{definition}
\bigskip

 We will prove that smaller balls induced by points should essentially be covered by the bigger balls in which they are embedded. 
 
\begin{remark}
In the geometry of balls induced under compression of scale $m>0$, we will implicitly assume that 
\begin{align}
0<m\leq 1.\nonumber
\end{align}
For simplicity, we will choose to write the ball induced by the point $\vec{x}=(x_1,x_2,\ldots,x_n)$ under compression as \begin{align}\mathcal{B}_{\frac{1}{2}\mathcal{G}\circ \mathbb{V}_m[\vec{x}]}[\vec{x}].\nonumber
\end{align}
We adopt this notation to save enough work space in many circumstances.
\end{remark}

\begin{proposition}\label{cgidentity}
Let $(x_1,x_2,\ldots, x_n)\in \mathbb{R}^n$ for $n\geq 2$ with $x_j\neq 0$ for $j=1,\ldots,n$. We have 
\begin{align}
\mathcal{G}\circ \mathbb{V}_m[(x_1,x_2,\ldots, x_n)]^2=\mathcal{M}\circ \mathbb{V}_1\bigg[\bigg(\frac{1}{x_1^2},\ldots,\frac{1}{x_n^2}\bigg)\bigg]+m^2\mathcal{M}\circ \mathbb{V}_1[(x_1^2,\ldots,x_n^2)]-2mn.\nonumber
\end{align}
In particular, if $m=m(n)=o(1)$ as $n\longrightarrow \infty$, then
\begin{align}
\mathcal{G}\circ \mathbb{V}_m[(x_1,x_2,\ldots, x_n)]^2=\mathcal{M}\circ \mathbb{V}_1\bigg[\bigg(\frac{1}{x_1^2},\ldots,\frac{1}{x_n^2}\bigg)\bigg]-2mn+O\bigg(m^2\mathcal{M}\circ \mathbb{V}_1[(x_1^2,\ldots,x_n^2)]\bigg)\nonumber
\end{align}
for $\vec{x}\in \mathbb{R}^n$ with $x_i\geq 1$ for each $1\leq i\leq n.$
\end{proposition}
\bigskip

Proposition \ref{cgidentity} offers an extremely useful identity. It allows for passage from the \emph{compression gap} of the points under compression to the relative distance to the origin. It suggests that points under compression with a large gap must be far away from the origin than points with a relatively smaller gap under compression. That is to say, the inequality 
\begin{align}
\mathcal{G}\circ \mathbb{V}_m[\vec{x}]< \mathcal{G}\circ \mathbb{V}_m[\vec{y}]\nonumber
\end{align}
with $m:=m(n)=o(1)$ as $n\longrightarrow \infty$ if and only if $||\vec{x}||\lesssim ||\vec{y}||$ for $\vec{x}, \vec{y}\in \mathbb{R}^n$ with $x_i\geq 1$ for all $1\leq i\leq n$. This important transference principle will be applied to obtain our results. In particular, we can write the asymptotic 
$$
\mathcal{G}\circ \mathbb{V}_m[(x_1,x_2,\ldots, x_n)]^2\sim \mathcal{M}\circ \mathbb{V}_1\bigg[\bigg(\frac{1}{x_1^2},\ldots,\frac{1}{x_n^2}\bigg)\bigg]=||\vec{x}||^2.
$$

\begin{corollary}\label{compression gap estimate 2}
Let $(x_1,x_2,\ldots, x_n)\in \mathbb{R}^n$ for $n\geq 2$ with $x_j\neq x_i$ for $j\neq i$ and $x_i,x_j\geq 1$ for each $1\leq i,j\leq n$. If $m:=m(n)=o(1)$ as $n\longrightarrow \infty$, then
\begin{align}
\mathcal{G}\circ \mathbb{V}_m[(x_1,x_2,\ldots, x_n)]^2\geq n\mathrm{inf}(x_j^2)-2mn+O\bigg(m^2\mathcal{M}\circ \mathbb{V}_1[(x_1^2,\ldots,x_n^2)]\bigg)\nonumber
\end{align}
and 
\begin{align}
\mathcal{G}\circ \mathbb{V}_m[(x_1,x_2,\ldots, x_n)]^2\leq n\mathrm{sup}(x_j^2)-2mn+O\bigg(m^2\mathcal{M}\circ \mathbb{V}_1[(x_1^2,\ldots,x_n^2)]\bigg).\nonumber
\end{align}
\end{corollary}

\begin{lemma}[Compression estimate]\label{gapestimate}
Let $(x_1,x_2,\ldots, x_n)\in \mathbb{R}^n$ for $n\geq 2$ with $x_i\geq 1$ for all $1\leq i\leq n$ with $x_i\neq x_j$~($i\neq j$). If $m:=m(n)=o(1)$ as $n\longrightarrow \infty$, then 
\begin{align}
\mathcal{G}\circ \mathbb{V}_m[(x_1,x_2,\ldots, x_n)]^2\ll n\mathrm{sup}(x_j^2)+m^2\log \bigg(1+\frac{n-1}{\mathrm{inf}(x_j)^2}\bigg)-2mn\nonumber
\end{align} 
and 
\begin{align}
\mathcal{G}\circ \mathbb{V}_m[(x_1,x_2,\ldots, x_n)]^2\gg n\mathrm{inf}(x_j^2)+m^2\log \bigg(1-\frac{n-1}{\mathrm{sup}(x_j^2)}\bigg)^{-1}-2mn.\nonumber
\end{align}
\end{lemma}
\bigskip

\begin{remark}
We observe that the inequality in Corollary \ref{compression gap estimate 2} implies the inequalities in Lemma \ref{gapestimate}. At any given moment, we will decide which of the versions of these inequalities to use. Indeed, the inequalities in Corollary \ref{compression gap estimate 2} are more applicable to various problems than the one in Lemma \ref{gapestimate}.
\end{remark}

\begin{theorem}\label{decider}
Let $\vec{z}=(z_1,z_2,\ldots,z_n)\in \mathbb{R}^n$ with $z_i\neq z_j$ for all $1\leq i<j\leq n$ with $y_i,z_i\geq 1$ for all $1\leq i\leq n$ and $m:=m(n)=o(1)$ as $n\longrightarrow \infty$. Then  $\vec{z}\in \mathcal{B}_{\frac{1}{2}\mathcal{G}\circ \mathbb{V}_m[\vec{y}]}[\vec{y}]$ with $||\vec{z}||<||\vec{y}||$ if and only if 
\begin{align}
\mathcal{G}\circ \mathbb{V}_m[\vec{z}]\leq \mathcal{G}\circ \mathbb{V}_m[\vec{y}]\nonumber
\end{align}
with $||\vec{y}-\vec{z}||<\epsilon$ for some $\epsilon>0$.
\end{theorem}

\begin{proof}
Let $\vec{z}\in \mathcal{B}_{\frac{1}{2}\mathcal{G}\circ \mathbb{V}_m[\vec{y}]}[\vec{y}]$ for $\vec{z}=(z_1,z_2,\ldots,z_n)\in \mathbb{R}^n$ with $z_i\neq z_j$ for all $1\leq i<j\leq n$ and $z_i\geq 1$ for all $1\leq i\leq n$ such that $||\vec{y}||>||\vec{z}||$. Suppose that 
\begin{align}
\mathcal{G}\circ \mathbb{V}_m[\vec{z}]>\mathcal{G}\circ \mathbb{V}_m[\vec{y}].\nonumber
\end{align}
We deduce that $||\vec{y}||\lesssim||\vec{z}||$. This inequality cannot hold. In this case, we can take $\epsilon:=\frac{1}{2}\mathcal{G}\circ \mathbb{V}_m[\vec{y}]$. Conversely, suppose \begin{align}
\mathcal{G}\circ \mathbb{V}_m[\vec{z}]\leq \mathcal{G}\circ \mathbb{V}_m[\vec{y}].\nonumber
\end{align}
 By Proposition \ref{cgidentity}, we get $||\vec{z}||\lesssim||\vec{y}||$. Under the requirement $||\vec{y}-\vec{z}||<\epsilon$ for some $\epsilon>0$, we obtain the inequality
\begin{align}
\bigg|\bigg|\vec{z}-\frac{1}{2}\bigg(y_1+\frac{m}{y_1},\ldots,y_n+\frac{m}{y_n}\bigg)\bigg|\bigg|&\leq \bigg|\bigg|\vec{y}-\frac{1}{2}\bigg(y_1+\frac{m}{y_1},\ldots,y_n+\frac{m}{y_n}\bigg)\bigg|\bigg|+\epsilon\nonumber \\&=\frac{1}{2}\mathcal{G}\circ \mathbb{V}_m[\vec{y}]+\epsilon \nonumber
\end{align}
with $m=m(n)=o(1)$ as $n\longrightarrow \infty$. Choosing $\epsilon>0$ sufficiently small, we deduce $\vec{z}\in \mathcal{B}_{\frac{1}{2}\mathcal{G}\circ \mathbb{V}_m[\vec{y}]}[\vec{y}]$.
\end{proof}
\bigskip

In the geometry of balls under compression, we will assume that $n$ is sufficiently large for $\mathbb{R}^n$. In this regime, we will always take the scale of compression $m:=m(n)=o(1)$ as $n\longrightarrow \infty.$

\begin{theorem}\label{ballproof}
Let $\vec{x}=(x_1,x_2,\ldots,x_n)\in \mathbb{R}^n$ with $x_i\neq x_j$ for all $1\leq i<j\leq n$ with $y_i,x_i\geq 1$ for each $1\leq i\leq n$. If $\vec{y}\in \mathcal{B}_{\frac{1}{2}\mathcal{G}\circ \mathbb{V}_m[\vec{x}]}[\vec{x}]$ with $||\vec{y}||<||\vec{x}||$ for $||\vec{y}-\vec{x}||<\delta$ for sufficiently small $\delta>0$, then 
\begin{align}
\mathcal{B}_{\frac{1}{2}\mathcal{G}\circ \mathbb{V}_m[\vec{y}]}[\vec{y}]\subseteq \mathcal{B}_{\frac{1}{2}\mathcal{G}\circ \mathbb{V}_m[\vec{x}]}[\vec{x}]\nonumber
\end{align}
for $m:=m(n)=o(1)$ as $n\longrightarrow \infty.$
\end{theorem}

\begin{proof}
 Let $\vec{y}\in \mathcal{B}_{\frac{1}{2}\mathcal{G}\circ \mathbb{V}_m[\vec{x}]}[\vec{x}]$ with $||\vec{y}||<||\vec{x}||$ for $||\vec{y}-\vec{x}||<\delta$. By Theorem \ref{decider}, we get $\mathcal{G}\circ \mathbb{V}_m[\vec{x}]\gtrsim \mathcal{G}\circ \mathbb{V}_m[\vec{y}]$ with $||\vec{y}-\vec{x}||<\delta$ for sufficiently small $\delta>0$. Consequently, the ball $\mathcal{B}_{\frac{1}{2}\mathcal{G}\circ \mathbb{V}_m[\vec{x}]}[\vec{x}]$ is slightly larger than the ball $\mathcal{B}_{\frac{1}{2}\mathcal{G}\circ \mathbb{V}_m[\vec{y}]}[\vec{y}]$ due to their compression gaps, and the latter ball does not contain the point $\vec{x}$ by construction. We deduce $||\mathbb{V}_m[\vec{y}]||>||\mathbb{V}_m[\vec{x}]||$ and 
\begin{align}
\mathcal{G}\circ \mathbb{V}_m[\mathbb{V}_m[\vec{y}]]&=\mathcal{G}\circ \mathbb{V}_m[\vec{y}]\nonumber \\&\lesssim \mathcal{G}\circ \mathbb{V}_m[\vec{x}]\nonumber \\&=\mathcal{G}\circ \mathbb{V}_m[\mathbb{V}_m[\vec{x}]]\nonumber
\end{align}
with $||\mathbb{V}_m[\vec{y}]-\mathbb{V}_m[\vec{x}]||<\epsilon$ for small $\epsilon>0$. This implies 
\begin{align}
\mathcal{B}_{\frac{1}{2}\mathcal{G}\circ \mathbb{V}_m[\vec{y}]}[\vec{y}]\subseteq \mathcal{B}_{\frac{1}{2}\mathcal{G}\circ \mathbb{V}_m[\vec{x}]}[\vec{x}].\nonumber
\end{align} 
\end{proof}

\begin{remark}
Theorem \ref{ballproof} suggests a nesting property of balls induced by points under compression
\end{remark}

\subsection{Interior points and the limit points of balls induced under compression}

In this section, we introduce the notion of an interior and the limit point of balls induced by points under compression. 

\begin{definition}
Let $\vec{y}=(y_1,y_2,\ldots,y_n)\in \mathbb{R}^n$ with $y_i\neq y_j$ for all $1\leq i<j\leq n$. A point $\vec{z}\in \mathcal{B}_{\frac{1}{2}\mathcal{G}\circ \mathbb{V}_m[\vec{y}]}[\vec{y}]$ is an \emph{interior} point if 
\begin{align}
\mathcal{B}_{\frac{1}{2}\mathcal{G}\circ \mathbb{V}_m[\vec{z}]}[\vec{z}]\subseteq \mathcal{B}_{\frac{1}{2}\mathcal{G}\circ \mathbb{V}_m[\vec{x}]}[\vec{x}]\nonumber
\end{align}
for most $\vec{x}\in \mathcal{B}_{\frac{1}{2}\mathcal{G}\circ \mathbb{V}_m[\vec{y}]}[\vec{y}]$. An interior point $\vec{z}$ is said to be a \emph{limit} point if 
\begin{align}
\mathcal{B}_{\frac{1}{2}\mathcal{G}\circ \mathbb{V}_m[\vec{z}]}[\vec{z}]\subseteq \mathcal{B}_{\frac{1}{2}\mathcal{G}\circ \mathbb{V}_m[\vec{x}]}[\vec{x}]\nonumber
\end{align}
for all $\vec{x}\in \mathcal{B}_{\frac{1}{2}\mathcal{G}\circ \mathbb{V}_m[\vec{y}]}[\vec{y}]$
\end{definition}
\bigskip

We will show that an interior and limit point exists in any ball induced by points under compression of any scale $m$ with $0<m\leq 1$ in any dimension.

\begin{theorem}\label{limitexistence}
Let $\vec{x}=(x_1,x_2,\ldots,x_n)\in \mathbb{R}^n$ with $x_i\neq x_j$ for all $1\leq i<j\leq n$ with $y_i\geq 1$ for all $1\leq i\leq n$. The ball $\mathcal{B}_{\frac{1}{2}\mathcal{G}\circ \mathbb{V}_m[\vec{x}]}[\vec{x}]$ contains an interior point and a limit point.
\end{theorem}

\begin{proof}
Let $\vec{x}=(x_1,x_2,\ldots,x_n)\in \mathbb{R}^n$ with $x_i\neq x_j$ for all $1\leq i<j\leq n$ with $x_i\geq 1$ for all $1\leq i\leq n$. Suppose that $\mathcal{B}_{\frac{1}{2}\mathcal{G}\circ \mathbb{V}_m[\vec{x}]}[\vec{x}]$ contains no limit point. Pick 
\begin{align}
\vec{z}_1\in \mathcal{B}_{\frac{1}{2}\mathcal{G}\circ \mathbb{V}_m[\vec{x}]}[\vec{x}]\nonumber
\end{align}
such that $z_{1_i}\geq 1$ for each $1\leq i\leq n$ with $||\vec{z}_1||<||\vec{x}||$ such that $||\vec{z}_1-\vec{x}||<\epsilon$ for sufficiently small $\epsilon>0$. By Theorem \ref{ballproof} and Theorem \ref{decider}, we get
\begin{align}
\mathcal{B}_{\frac{1}{2}\mathcal{G}\circ \mathbb{V}_m[\vec{z}_1]}[\vec{z}_1]\subset \mathcal{B} _{\frac{1}{2}\mathcal{G}\circ \mathbb{V}_m[\vec{x}]}[\vec{x}]\nonumber
\end{align}
with $\mathcal{G}\circ \mathbb{V}_m[\vec{z}_1]\lesssim \mathcal{G}\circ \mathbb{V}_m[\vec{x}]$. Again, we pick $\vec{z}_2\in \mathcal{B}_{\frac{1}{2}\mathcal{G}\circ \mathbb{V}_m[\vec{z}_1]}[\vec{z}_1]$ such that $z_{2_i}\geq 1$ for each $1\leq i\leq n$ with $||\vec{z}_2||<||\vec{z}_1||$ such that $||\vec{z}_2-\vec{z}_1||<\delta$ for sufficiently small $\delta>0$. By Theorem \ref{ballproof} and Theorem \ref{decider}, we deduce
\begin{align}
\mathcal{B}_{\frac{1}{2}\mathcal{G}\circ \mathbb{V}_m[\vec{z}_2]}[\vec{z}_2]\subset \mathcal{B}_{\frac{1}{2}\mathcal{G}\circ \mathbb{V}_m[\vec{z}_1]}[\vec{z}_1]\nonumber
\end{align}
with $\mathcal{G}\circ \mathbb{V}_m[\vec{z}_2]\lesssim \mathcal{G}\circ \mathbb{V}_m[\vec{z}_1]$. Continuing the argument in this way,  we obtain the infinite descending sequence of the compression gaps
\begin{align}
\mathcal{G}\circ \mathbb{V}_m[\vec{x}]\gtrsim \mathcal{G}\circ \mathbb{V}_m[\vec{z}_1]\gtrsim \mathcal{G}\circ \mathbb{V}_m[\vec{z}_2]\gtrsim \cdots \gtrsim \mathcal{G}\circ \mathbb{V}_m[\vec{z}_n]\gtrsim \cdots.\nonumber
\end{align}
\end{proof}

\begin{proposition}\label{found limit point}
The point $\vec{x}=(x_1,x_2,\ldots,x_n)$ with $x_i=1$ for each $1\leq i\leq n$ is the limit point of the ball $\mathcal{B}_{\frac{1}{2}\mathcal{G}\circ \mathbb{V}_1[\vec{y}]}[\vec{y}]$ for any $\vec{y}=(y_1,y_2,\ldots,y_n)\in \mathbb{R}^n$ with $y_i>1$ for each $1\leq i\leq n$.
\end{proposition}

\begin{proof}
Applying the compression $\mathbb{V}_1:\mathbb{R}^n\longrightarrow \mathbb{R}^n$ on the point $\vec{x}=(x_1,x_2,\ldots,x_n)$ with $x_i=1$ for each $1\leq i\leq n$, we obtain $\mathbb{V}_1[\vec{x}]=(1,1,\ldots,1)$ so that $\mathcal{G}\circ \mathbb{V}_1[\vec{x}]=0$ and the corresponding ball induced under compression $\mathcal{B}_{\frac{1}{2}\mathcal{G}\circ \mathbb{V}_1[\vec{x}]}[\vec{x}]$ contain only the point $\vec{x}$. By definition \ref{limitexistence}, the point $\vec{x}$ must be the limit point of the ball $\mathcal{B}_{\frac{1}{2}\mathcal{G}\circ \mathbb{V}_1[\vec{x}]}[\vec{x}]$. We deduce
\begin{align}
 \mathcal{B}_{\frac{1}{2}\mathcal{G}\circ \mathbb{V}_1[\vec{x}]}[\vec{x}]\subseteq \mathcal{B}_{\frac{1}{2}\mathcal{G}\circ \mathbb{V}_1[\vec{y}]}[\vec{y}]\nonumber 
\end{align}
for any $\vec{y}=(y_1,y_2,\ldots,y_n)\in \mathbb{R}^n$ with $y_i>1$ for all $1\leq i\leq n$. If we assume 
\begin{align}
 \mathcal{B}_{\frac{1}{2}\mathcal{G}\circ \mathbb{V}_1[\vec{x}]}[\vec{x}] \not \subseteq  \mathcal{B}_{\frac{1}{2}\mathcal{G}\circ \mathbb{V}_1[\vec{y}]}[\vec{y}]\nonumber    
\end{align}
holds for some $\vec{y}=(y_1,y_2,\ldots,y_n)\in \mathbb{R}^n$ with $y_i>1$ for each $1\leq i\leq n$, then there must exist some point $\vec{z}\in \mathcal{B}_{\frac{1}{2}\mathcal{G}\circ \mathbb{V}_1[\vec{x}]}[\vec{x}]$ such that $\vec{z}\not \in \mathcal{B}_{\frac{1}{2}\mathcal{G}\circ \mathbb{V}_1[\vec{y}]}[\vec{y}]$. Since $\vec{x}$ is the only point in the ball $\mathcal{B}_{\frac{1}{2}\mathcal{G}\circ \mathbb{V}_1[\vec{x}]}[\vec{x}]$, we get
\begin{align}
    \vec{x}\not \in \mathcal{B}_{\frac{1}{2}\mathcal{G}\circ \mathbb{V}_1[\vec{y}]}[\vec{y}]\nonumber
\end{align}
which is inconsistent with the fact that $\vec{x}$ is the limit point of the ball.
\end{proof}

\subsection{Admissible points of balls induced under compression}

We introduce the notion of \emph{admissible} points of balls induced by points under compression. 

\begin{definition}
Let $\vec{y}=(y_1,y_2,\ldots,y_n)\in \mathbb{R}^n$ with $y_i\neq y_j$ for all $1\leq i<j\leq n$. The point $\vec{y}$ is said to be an \emph{admissible} point of the ball $\mathcal{B}_{\frac{1}{2}\mathcal{G}\circ \mathbb{V}_m[\vec{x}]}[\vec{x}]$ if \begin{align}
\bigg|\bigg|\vec{y}-\frac{1}{2}\bigg(x_1+\frac{m}{x_1},\ldots,x_n+\frac{m}{x_n}\bigg)\bigg|\bigg|=\frac{1}{2}\mathcal{G}\circ \mathbb{V}_m[\vec{x}].\nonumber
\end{align}
\end{definition}
\bigskip

The admissible points of the balls induced by points under compression encompasses points on the ball. Geometrically, they sit on the outer part of the induced ball. Here, we will show that all balls can in principle be generated by their admissible points.

\begin{theorem}\label{admissibletheorem}
Let $\vec{x}\in \mathbb{R}^n$ with $x_i\neq x_j$~($i\neq j$) such that $x_i,y_i \geq 1$ for all $1\leq i\leq n$ and set $m:=m(n)=o(1)$ as $n\longrightarrow \infty$. The point $\vec{y}\in \mathcal{B}_{\frac{1}{2}\mathcal{G}\circ \mathbb{V}_m[\vec{x}]}[\vec{x}]$ with $||\vec{y}||<||\vec{x}||$ such that $||\vec{y}-\vec{x}||<\epsilon$ for sufficiently small $\epsilon>0$ is admissible if and only if 
\begin{align}
\mathcal{B}_{\frac{1}{2}\mathcal{G}\circ \mathbb{V}_m[\vec{y}]}[\vec{y}]=\mathcal{B}_{\frac{1}{2}\mathcal{G}\circ \mathbb{V}_m[\vec{x}]}[\vec{x}]\nonumber
\end{align}
and $\mathcal{G}\circ \mathbb{V}_m[\vec{y}]=\mathcal{G}\circ \mathbb{V}_m[\vec{x}]$.
\end{theorem}

\begin{proof}
We let  $\vec{y}\in \mathcal{B}_{\frac{1}{2}\mathcal{G}\circ \mathbb{V}_m[\vec{x}]}[\vec{x}]$ with $||\vec{y}||<||\vec{x}||$ so that $||\vec{y}-\vec{x}||<\epsilon$ for sufficiently small $\epsilon>0$  be admissible. Suppose that 
\begin{align}
\mathcal{B}_{\frac{1}{2}\mathcal{G}\circ \mathbb{V}_m[\vec{y}]}[\vec{y}]\neq \mathcal{B}_{\frac{1}{2}\mathcal{G}\circ \mathbb{V}_m[\vec{x}]}[\vec{x}].\nonumber
\end{align}
Without loss of generality, we can choose some $\vec{z}\in \mathcal{B}_{\frac{1}{2}\mathcal{G}\circ \mathbb{V}_m[\vec{x}]}[\vec{x}]$ with $||\vec{z}||<||\vec{x}||$ such that 
\begin{align}
\vec{z}\notin \mathcal{B}_{\frac{1}{2}\mathcal{G}\circ \mathbb{V}_m[\vec{y}]}[\vec{y}].\nonumber
\end{align}
for $||\vec{z}-\vec{x}||<\delta$ for sufficiently small $\delta>0$. Applying Theorem \ref{decider}, we get 
\begin{align}
\mathcal{G}\circ \mathbb{V}_m[\vec{y}]\lesssim \mathcal{G}\circ \mathbb{V}_m[\vec{x}].\nonumber
\end{align}
This already violates the equality $\mathcal{G}\circ \mathbb{V}_m[\vec{y}]=\mathcal{G}\circ \mathbb{V}_m[\vec{x}]$. The latter equality of compression gaps follows from the requirement that the balls are indistinguishable. Conversely, suppose 
\begin{align}
\mathcal{B}_{\frac{1}{2}\mathcal{G}\circ \mathbb{V}_m[\vec{y}]}[\vec{y}]=\mathcal{B}_{\frac{1}{2}\mathcal{G}\circ \mathbb{V}_m[\vec{x}]}[\vec{x}]\nonumber
\end{align}
and $\mathcal{G}\circ \mathbb{V}_m[\vec{y}]=\mathcal{G}\circ \mathbb{V}_m[\vec{x}]$. It follows that the point $\vec{y}$ lives on the outer part of the two indistinguishable balls and so must satisfy the equality
\begin{align}
\bigg|\bigg|\vec{z}-\frac{1}{2}\bigg(y_1+\frac{m}{y_1},\ldots,y_n+\frac{m}{y_n}\bigg)\bigg|\bigg|&=\bigg|\bigg|\vec{z}-\frac{1}{2}\bigg(x_1+\frac{m}{x_1},\ldots,x_n+\frac{m}{x_n}\bigg)\bigg|\bigg|\nonumber \\&=\frac{1}{2}\mathcal{G}\circ \mathbb{V}_m[\vec{x}].\nonumber
\end{align}
We deduce
\begin{align}
\frac{1}{2}\mathcal{G}\circ \mathbb{V}_m[\vec{x}]&=\bigg|\bigg|\vec{y}-\frac{1}{2}\bigg(x_1+\frac{m}{x_1},\ldots,x_n+\frac{m}{x_n}\bigg)\bigg|\bigg|\nonumber
\end{align}
and the point $\vec{y}$ is admissible.
\end{proof}
\bigskip

\begin{proposition}\label{No three collinear}
No three admissible points on the ball $\mathcal{B}_{\frac{1}{2}\mathcal{G}\circ \mathbb{V}_m[\vec{x}]}[\vec{x}]$ are collinear.
\end{proposition}

\section{Main result}\label{sec:main}

In this section, we prove the main result of this paper.

\begin{theorem}
The number of points that can be placed in the grid $n\times n\times \cdots \times n~(d~times)=n^d$ for all $d\in \mathbb{N}$ and with $d\geq 2$ such that no three points are collinear satisfies the lower bound
\begin{align}
\gg n^{d-1}\sqrt[2d]{d}.\nonumber
\end{align}
\end{theorem}

\begin{proof}
Let $m:=m(d)=o(1)$ as $d\longrightarrow \infty$ and choose a point $\vec{x}\in \mathbb{R}^d$ such that $\mathcal{G}\circ \mathbb{V}_1[\vec{x}]\sim n^d$ for a fixed $n$. We note that such a point exists; that is, we choose $\vec{x}$ such that the largest coordinate $\mathrm{sup}(x_i)_{i=1}^{d}=n^d+o(1)$ and the smallest coordinate $\mathrm{inf}(x_i)_{i=1}^{d}\geq 1$. We apply the compression  $\mathbb{V}_m$ on the point $\vec{x}$ and construct the induced ball
\begin{align}
\mathcal{B}_{\frac{1}{2}\mathcal{G}\circ \mathbb{V}_m[\vec{x}]}[\vec{x}].\nonumber
\end{align}
Due to the restriction $\mathcal{G}\circ \mathbb{V}_m[\vec{x}]\sim n^d$ all admissible points $\vec{x_k}$ for $\vec{x_k}\neq \vec{x}$ with $||\vec{x}_k-\vec{x}||<\epsilon$ for small $\epsilon>0$ on the ball have the property that 
\begin{align}
\mathcal{G}\circ \mathbb{V}_m[\vec{x}]=\mathcal{G}\circ \mathbb{V}_m[\vec{x_k}]\sim n^d\nonumber
\end{align}
with 
$$
\mathcal{B}_{\frac{1}{2}\mathcal{G}\circ \mathbb{V}_m[\vec{x}]}[\vec{x}]=\mathcal{B}_{\frac{1}{2}\mathcal{G}\circ \mathbb{V}_m[\vec{x}_k]}[\vec{x}_k]
$$ 
by virtue of Theorem \ref{admissibletheorem}. Again for the ball $\mathcal{B}_{\frac{1}{2}\mathcal{G}\circ \mathbb{V}_m[\vec{x}_k]}[\vec{x}_k]$, we choose an admissible point $\vec{x}_l$ such that $\vec{x}_l\neq \vec{x}_k$ with $||\vec{x}_l-\vec{x}_k||<\epsilon$ for the same choice of small $\epsilon>0$ but with $||\vec{x}_l-\vec{x}||\geq \epsilon$. We have 
\begin{align}
\mathcal{G}\circ \mathbb{V}_m[\vec{x}_l]=\mathcal{G}\circ \mathbb{V}_m[\vec{x}_k]\sim n^d\nonumber
\end{align}
with 
$$
\mathcal{B}_{\frac{1}{2}\mathcal{G}\circ \mathbb{V}_m[\vec{x}_l]}[\vec{x}_l]=\mathcal{B}_{\frac{1}{2}\mathcal{G}\circ \mathbb{V}_m[\vec{x}_k]}[\vec{x}_k]
$$ 
by virtue of Theorem \ref{admissibletheorem}. This process can be iterated. It is seen that for any admissible point $\vec{x}_j$ on the ball so constructed $\mathcal{B}_{\frac{1}{2}\mathcal{G}\circ \mathbb{V}_m[\vec{x}]}[\vec{x}]$, we must have 
$$
\mathcal{G}\circ \mathbb{V}_m[\vec{x}_j]\sim n^d.
$$ 
We note that we only need to carry out this iteration for points with coordinates $\geq 1$, since compression images of admissible points are also admissible. We construct the smallest $d$ dimensional box that contains this ball. In this box, we construct the $n\times n\times \cdots \times n~(d~times)=n^d$ grid. We only consider admissible points of the ball that are on the constructed grid. In the grid $n\times n\times \cdots \times n~(d~times)=n^d$ for all $d\in \mathbb{N}$ with $d\geq 2$, the number of points that can be arranged in such a way that no three are collinear can be lower bounded by counting only the number of admissible points on the ball so constructed and on the grid constructed by Proposition \ref{No three collinear}. We deduce the lower bound
\begin{align}
&\geq \sum \limits_{\substack{\vec{x_j}\in n^d\\\sqrt[d]{\mathcal{G}\circ \mathbb{V}_m[\vec{x_j}]}=n}}1\nonumber \\&=\sum \limits_{\substack{\vec{x_j}\in n^d\\\mathrm{inf}(x_{j_i})_{i=1}^{d}\geq 1\\\mathrm{sup}(x_{j_i})_{i=1}^{d}=n^d+o(1)}}\frac{\sqrt[d]{\mathcal{G}\circ \mathbb{V}_m[\vec{x_j}]}}{n}\nonumber \\&\gg \frac{1}{n}\sum \limits_{\substack{\vec{x_j}\in n^d\\\mathrm{inf}(x_{j_i})_{i=1}^{d}\geq 1\\\mathrm{sup}(x_{j_i})_{i=1}^{d}=n^d+o(1)}}\sqrt[2d]{d}\sqrt[2d]{(\mathrm{inf}(x_{j_i})_{i=1}^{d})^2}\nonumber \\&\geq \frac{\sqrt[2d]{d}}{n}\sum \limits_{\vec{x_j}\in n^d}1\nonumber \\&=n^{d-1}\sqrt[2d]{d}.\nonumber
\end{align}
The claimed lower bound follows as a consequence.
\end{proof}

\begin{corollary}
The number of points that can be placed in the grid $n\times n$ such that no three points are collinear satisfies the lower bound
\begin{align}
\geq nC_1\sqrt[4]{2}\nonumber
\end{align}
for some $C_1>0$.
\end{corollary}
\bigskip

\begin{corollary}
The number of points that can be placed in the grid $n\times n\times n$ such that no three points are collinear satisfies the lower bound
\begin{align}
\geq n^2C_2\sqrt[6]{3}\nonumber
\end{align}
for some $C_2>0$.
\end{corollary}

\footnote{
\par
.}%
.

\bibliographystyle{amsplain}

\end{document}